\documentclass[12pt]{amsart} 
\usepackage{a4wide}
\usepackage{amssymb}
\usepackage{amsthm}
\usepackage{amsmath}
\usepackage{amscd}   
\usepackage{verbatim}
\usepackage{color}
\usepackage{hyperref}
\textheight8.65in \textwidth6.5in \numberwithin{equation}{section}

\theoremstyle{plain}
\newtheorem{theorem}{Theorem}[section]
\newtheorem{corollary}[theorem]{Corollary}
\newtheorem{lemma}[theorem]{Lemma}
\newtheorem{proposition}[theorem]{Proposition}

\theoremstyle{definition}

\theoremstyle{remark}
\newtheorem*{remark}{Remark}

\newcommand{\Dhat}{\widehat{D}}
\newcommand{\R}{\mathbb{R}}
\newcommand{\Q}{\mathbb{Q}}
\newcommand{\Z}{\mathbb{Z}}
\newcommand{\N}{\mathbb{N}}
\newcommand{\C}{\mathbb{C}}

\renewcommand{\L}{\mathbb{L}}

\renewcommand{\P}{\mathbb{P}}

\newcommand{\calE}{\mathcal{E}}

\newcommand{\calF}{\mathcal{F}}

\newcommand{\calL}{\mathcal{L}}
\newcommand{\calM}{\mathcal{M}}

\newcommand{\calQ}{\mathcal{Q}}

\newcommand{\HH}{\mathbb{H}}

\newcommand{\Orth}{\operatorname{O}}

\newcommand{\SL}{{\text {\rm SL}}}

\begin{document}

\title[$p$-adic properties of modular shifted convolution $L$-series]{$p$-adic properties of modular shifted convolution Dirichlet series}

\author{Kathrin Bringmann, Michael H. Mertens, and Ken Ono}

\address{Mathematisches Institut der Universit\"at zu K\"oln, Weyertal 86-90,
D-50931 K\"oln, Germany} \email{kbringma@math.uni-koeln.de}

\address{Mathematisches Institut der Universit\"at zu K\"oln, Weyertal 86-90,
D-50931 K\"oln, Germany} \email{mmertens@math.uni-koeln.de}

\address{Department of Mathematics and Computer Science, Emory University,
Atlanta, Georgia 30022} \email{ono@mathcs.emory.edu}
\thanks{
The research of the first author was supported by the Alfried Krupp Prize for Young University Teachers of the Krupp foundation and the research leading to these results has received funding from the European Research Council under the European Union's Seventh Framework Programme (FP/2007-2013) / ERC Grant agreement n. 335220 - AQSER. The second author thanks the DFG-Graduiertenkolleg 1269 `Global Structures in Geometry and Analysis' for the financial support of his research. The third author thanks the National Science Foundation and the Asa Griggs Candler Fund for their generous support.}


\begin{abstract}
Hoffstein and Hulse recently introduced the notion of {\it shifted convolution Dirichlet series}
for pairs of modular forms $f_1$ and $f_2$. The second two authors investigated certain special values of {\it symmetrized sums} of such functions, numbers which are generally expected to be mysterious
transcendental numbers.
They proved that the generating functions
of these values in the $h$-aspect are linear combinations of  mixed mock modular forms and quasimodular forms. Here we examine the special cases when $f_1=f_2$ where, in addition, there is a prime $p$ for which $p^2$ divides the level. We prove that the mixed mock modular form is a linear combination of at most two  weight 2 weakly holomorphic $p$-adic modular forms.
\end{abstract}

\maketitle

\section{Introduction and statement of results}

Suppose that $f_1, f_2 \in S_{k}(\Gamma_0(N))$ 
 are even integer weight $k$ cusp forms with $L$-functions
$$
L(f_j,s)=\sum_{n=1}^{\infty}\frac{a_j(n)}{n^s}.
$$
Rankin and Selberg \cite{Rankin, Selberg} independently introduced the so-called {\it Rankin-Selberg convolution}
$$
L(f_1\otimes f_2,s):=\sum_{n=1}^{\infty}\frac{a_{1}(n)\overline{a_{2}(n)}}{n^s},
$$
one of the most fundamental objects in the theory of automorphic forms.
Later in 1965, Selberg \cite{Selberg2} introduced {\it shifted convolution} $L$-functions,
series which  play an important role in progress towards Ramanujan-type conjectures for Fourier coefficients and the
Lindel\"of Hypothesis for
automorphic $L$-functions inside the critical strip.

In a recent paper, Hoffstein and Hulse
 \cite{HoffsteinHulse}\footnote{Here we choose slightly different normalizations
for Dirichlet series from those that appear in \cite{HoffsteinHulse}.} introduced the
 {\it shifted convolution
series}
\begin{equation}\label{shiftedseries}
D(f_1,f_2,h;s):=\sum_{n=1}^{\infty}\frac{a_{1}(n+h)\overline{a_{2}(n)}}{n^s}.
\end{equation}
They obtained the meromorphic continuation of these series and certain multiple Dirichlet series
which are obtained by additionally summing in $h$ aspect. 

The second two authors recently investigated various symmetrized forms of these Dirichlet series \cite{MO14}, and
they proved that the generating functions of certain special values in the $h$-aspect are sums of mixed mock modular forms and
quasimodular forms.

Here we study the special case where $f_1=f_2$ is an even integer weight newform.
For each positive integer $h$  we define the {\it symmetrized shifted convolution}
Dirichlet series
\begin{equation}\label{symmetrizedseries}
\Dhat(f,f,h;s):=D(f,f,h;s)-D\left(\overline{f},\overline{f},-h;s\right).
\end{equation}
If $f$ has weight $k$, then we define the generating function
\begin{equation}\label{LDefinition}
\L(f,f;\tau):=\sum_{h=1}^{\infty}\Dhat(f,f,h;k-1)q^h,
\end{equation}
where $q:=e^{2\pi i \tau}$ and $\tau\in\HH$, the upper-half of the complex plane.
In \cite{MO14}, the second two authors proved that $\L(f,f;\tau)$ is the sum of a weight 2 mixed mock modular form and a weight 2 quasimodular form.

For example, if
$f=\Delta$, the normalized weight 12 cusp form on $\SL_2(\Z)$, then we have that
\begin{equation}\label{DeltaExample}
\L(\Delta,\Delta;\tau)=-33.383\dots q + 266.439\dots q^2 - 1519.218\dots q^3+ 4827.434\dots q^4-\dots\: .
\end{equation}
Using the usual Eisenstein series $E_{2k}=E_{2k}(\tau)$ and Klein's $j$-function, we let
$$
\sum_{n=-1}^{\infty} r(n)q^n:=-\Delta\left(j^2-1464j-\alpha^2+1464\alpha\right),
$$
where $\alpha=106.10455\dots$. If $\beta=2.8402\dots$, then Theorem 1.1 of
\cite{MO14} is illustrated by the identity
\begin{displaymath}
\L(\Delta,\Delta;\tau)=-\frac{\Delta}{\beta}\left (\frac{65520}{691}-\sum_{n\neq 0} r(n)n^{-11}q^n\right)
-\frac{E_2}{\beta}.
\end{displaymath}
The first summand above is a weight 2  {\it mixed mock modular form}; it is essentially the product of the weight 12 modular form $\Delta$ with the weight $-10$ mock modular form
(see Section~\ref{HMF})
\begin{equation*}
\frac{65520}{691}-\sum_{n\neq 0} r(n)n^{-11}q^n.
\end{equation*}
The second summand $-E_2/\beta$ is a weight 2 quasimodular form.

The numerics in  (\ref{DeltaExample}) provide evidence for the general belief that
such special values are mysterious transcendental numbers.  In this case, their algebraic properties are  dictated by the
constants $\alpha$ and $\beta$. In particular, these constants are the only sources of irrationality for the coefficients in the generating function. Therefore, to better understand these special values, it is natural to investigate situations where these $q$-series are nearly algebraic, and to then
investigate their $p$-adic properties.
In this note we identify one particular situation where these difficulties can be addressed.

To make this precise, we make use of Eichler integrals, $p$-adic modular forms, and weight 2 weakly holomorphic quasimodular forms. If $F(\tau)=\sum_{n\in \Z} A(n)q^n$ is a weight $k$ {\it weakly holomorphic modular form},
one whose poles (if any) are supported at cusps, then its {\it Eichler integral} is 
\begin{equation*}
\mathcal{E}_F(\tau):=\sum_{n\neq 0} A(n)n^{1-k}q^n.
\end{equation*}
Following the seminal work\footnote{The notion used here is slightly different from Serre's original definition. We do not require that $p$-adic modular forms are limits of holomorphic modular forms.} of Serre  \cite{Serre},
we say that a $q$-series $G(q)$ is a {\it weakly holomorphic $p$-adic modular form} of weight $k$ if there exists a sequence of weakly holomorphic modular forms whose coefficients $p$-adically tend to those of $G(q)$
, with the additional property that their weights $p$-adically tend to $k$.
Finally, we require weight 2 weakly holomorphic quasimodular forms. If $M^{!}_2(\Gamma_0(N))$ denotes the space of weight 2 weakly holomorphic modular forms on $\Gamma_0(N)$, then
the space of {\it weight 2 weakly holomorphic quasimodular forms} is defined as
\begin{equation*}
\widetilde{M}_2^{!}(\Gamma_0(N)):= \C E_2 \oplus M^{!}_2(\Gamma_0(N)).
\end{equation*}

For even weight newforms $f$ whose level is divisible by the square of a prime $p$, we obtain the following theorem about
the $p$-adic properties of $\L(f,f;\tau)$.

\begin{theorem}\label{main}
Let $f\in S_k(\Gamma_0(N))$ be an even weight  newform.
 If $p$ is a prime with $p^2\mid N$, then there exist constants $\delta_1,
 \delta_2\in\C$, a weight 2 weakly holomorphic quasimodular form $\calQ_f\in\widetilde{M}^!_2(\Gamma_0(N))$, and a weight $2-k$ weakly holomorphic $p$-adic modular form $\calL_f$ for which
\begin{equation*}
\L(f,f;\tau)=\delta_1f(\tau)\calL_f(\tau)+\delta_2 f(\tau)\mathcal{E}_f(\tau) +\calQ_f (\tau).
\end{equation*}
Moreover, if $f$ has complex multiplication, then
 there are choices with $\delta_2=0$.
\end{theorem}

\begin{remark}
The proof of Theorem~\ref{main} shows that $ \mathcal{E}_f$ is also a weight $2-k$ $p$-adic modular form. In fact, it is a weight $2-k$ cuspidal $p$-adic modular form, a $p$-adic limit of cusp forms whose weights have $p$-adic limit $2-k$.
Therefore, the mixed mock modular form $\delta_1 f\calL_f+\delta_2 f\mathcal{E}_f$ is the linear combination of a weight 2 weakly holomorphic $p$-adic modular forms and a weight
2 cuspidal $p$-adic modular form. Moreover, if $f$ has CM, then $\L(f,f;\tau)$ is a linear combination of  the weight 2 weakly holomorphic $p$-adic modular form $f\mathcal{L}_f$ and
a quasimodular form $\calQ_f$.
\end{remark}

The proof of Theorem~\ref{main} makes use of the theory of harmonic Maass forms and the earlier work of the second two authors  \cite{MO14}. Theorem 1.1 of \cite{MO14}
implies that $\L(f,f;\tau)$ is a linear combination of a mixed mock modular form and a weight 2 weakly holomorphic quasimodular form. Therefore, to prove Theorem~\ref{main},
it suffices to show that these mixed mock modular forms can be decomposed as a linear combination of a weight 2 weakly holomorphic $p$-adic modular form and
 $f\mathcal{E}_f$. We make use of previous work on the algebraic normalizations of mock modular forms
 \cite{BOR, GKO} and the theory of
various types of Poincar\'e series (see Section~\ref{PoincareSeries}) to establish the existence of these linear combinations. The relationships between the relevant Poincar\'e series under the differential operators $D^{k-1}$ and $\xi_{2-k}$ (see Section~\ref{HMF})
plays a critical role in the proof of the theorem.
The hypotheses in Theorem~\ref{main}, combined with the theory of newforms,  implies that the coefficients of the relevant Poincar\'e series  vanish for those exponents that are
divisible by $p$. This fact depends on an elementary lemma about Kloosterman sums (see Lemma~\ref{Kloost}). These results, combined with some of the more elementary features of Serre's theory of $p$-adic modular forms,
proves that the mixed mock modular form can be viewed as a weight 2 $p$-adic modular form, which completes the proof of Theorem~\ref{main}.
In Section~\ref{example} we give a detailed example illustrating Theorem~\ref{main} and some its consequences concerning the $3$-adic properties of the symmetrized
shifted convolution special values for the unique weight 4 newform on $\Gamma_0(9)$.

\section{Harmonic Maass forms}\label{HMF}
In this brief section we recall the definition and most important facts about harmonic weak Maass forms. For more detailed information we refer the reader to \cite{Ono08,ZagierBourbaki} and the references therein. Throughout, we denote by $\HH$ the upper-half of the complex plane, and we write
 $\tau= x + iy \in \HH$, while $k$ is always an even integer. With this notation, we introduce the weight $k$ \emph{hyperbolic Laplacian}
\[\Delta_k:=-y^2\left(\frac{\partial^2}{\partial x^2}+\frac{\partial^2}{\partial y^2}\right)+iky\left(\frac{\partial}{\partial x}+i\frac{\partial}{\partial y}\right).\]
Recall the definition of the weight $k$ \emph{slash operator}. For a function $f:\HH\rightarrow\C$ and $\gamma = \left(\begin{smallmatrix}
                                                                                                                   a & b \\ c & d 
                                                                                                                 \end{smallmatrix}\right)	
\in\SL_2(\Z)$, we let
\[f|_k\gamma(\tau):=(c\tau+d)^{-k}f\left(\frac{a\tau+b}{c\tau+d}\right).
\]
                                    
A smooth function $f:\HH\rightarrow\C$ is called a \emph{harmonic (weak) Maass form}\footnote{We often omit the word `weak'.} of weight $2-k\in 2\Z$ on $\Gamma_0(N)$ if the following conditions hold:
\begin{enumerate}
\item $f|_{2-k}\gamma=f$ for all $\gamma\in\Gamma_0(N)$;
\item $\Delta_{2-k}(f)=0$;
\item There exists a polynomial $P_f(q)=\sum_{n\leq 0} c_f^{+}(n)q^n\in \C[q^{-1}]$ such that
$f(\tau)-P_f(q)=O(e^{-\varepsilon y})$ as $y\rightarrow +\infty$ for some $\varepsilon >0$. Analogous conditions are required at all cusps.
\end{enumerate}
The vector space of harmonic Maass forms of weight $2-k$ on $\Gamma_0(N)$ is denoted by $H_{2-k}(\Gamma_0(N))$.

\begin{remark} Weight $2-k$ weakly holomorphic modular forms are annihilated by $\Delta_{2-k}$, and so we naturally
have that $M_{2-k}^!(\Gamma_0(N))\subset H_{2-k}(\Gamma_0(N))$.
\end{remark}

The Fourier expansions of such forms split into a {\it holomorphic part} and a  {\it non-holomorphic part}.
\begin{lemma}
A harmonic Maass form $f$ of weight $2-k$ has a splitting
\begin{displaymath}
f(\tau)=f^+(\tau)+f^-(\tau),
\end{displaymath}
where for some $m_0\in\Z$ and $n_0\in\N$
we have the Fourier expansions
\[f^+(\tau):=\sum\limits_{n=m_0}^\infty c_f^+(n)q^n,\]
and
\[f^-(\tau):=\sum\limits_{\substack{n=n_0}}^\infty \overline{c_f^-(n)}n^{k-1}\Gamma(1-k;4\pi ny)q^{-n},\]
where $\Gamma(\alpha;x)$ denotes the usual incomplete Gamma-function.
\end{lemma}

The function $f^{+}$ (resp. $f^{-}$) is referred to as the {\it holomorphic part}  (resp. {\it nonholomorphic part}) of the harmonic Maass form $f$.
The holomorphic part of a harmonic Maass form is called a \emph{mock modular form} whenever $f^{-}$ is nontrivial.
In many contexts one encounters products of mock modular forms with (weakly holomorphic) modular forms of some fixed weight. These objects and their linear combinations are called \emph{mixed mock modular forms}.

Harmonic Maass forms are related naturally to classical elliptic modular forms by means of various differential operators.
The following lemma includes one of the most fundamental relationships concerning the
 $\xi$-operator which was introduced Bruinier and Funke (see Proposition 3.2 and Theorem 3.7 of \cite{BF04}).
\begin{proposition} If $f\in H_{2-k}(\Gamma_0(N))$, then
$ \xi_{2-k}(f):=2iy^{2-k}\overline{\frac{\partial f}{\partial\overline{\tau}}}$ is a weight $k$ cusp form on $\Gamma_0(N)$.
The resulting map $\xi_{2-k}:H_{2-k}(\Gamma_0(N))\rightarrow S_k(\Gamma_0(N))$
is surjective,  and it has kernel $M_{2-k}^!(\Gamma_0(N))$. Moreover, we have that
\[\xi_{2-k}(f)(\tau)=-(4\pi)^{k-1}\sum\limits_{n=n_0}^\infty c_f^-(n)q^n.\]
\end{proposition}

\begin{remark}
The cusp form $-(4\pi)^{1-k}\xi_{2-k}(f)$ is called the \emph{shadow} of the mock modular form $f^+$. Note that in the literature the normalization of the shadow may differ from the one chosen here.
\end{remark}

\section{Poincar\'{e} series}\label{PoincareSeries}

\subsection{The Poincar\'e series $P(m,k,N)$ and $Q(-m,k,N)$}
One very explicit way of constructing harmonic Maass forms with given (cuspidal) shadow uses Poincar\'e series, see for example \cite{BringOPNAS}. In this section, we review the basic facts about them which we will use later.

For $m\in\Z$, let $\varphi_m:\R^+ \rightarrow \C$ be a function which is $O(y^\alpha)$ for some $\alpha>0$ as $y\rightarrow 0$. Define the function $\varphi_m^*(\tau):=\varphi_m(y)\exp(2\pi imx)$. Further let
$\Gamma_\infty:=\left\{\pm \left(\begin{smallmatrix} 1 & n \\ 0 & 1\end{smallmatrix}\right) \ :\ n\in\Z\right\}.$
Obviously, the function $\varphi_m^*$ is invariant under the action of $\Gamma_\infty$. We define the general Poincar\'e series $\P(m,k,N,\varphi_m;\tau)$ by
\begin{equation*}
\P(m,k,N,\varphi_m;\tau):=\sum\limits_{\gamma\in\Gamma_\infty\setminus\Gamma_0(N)}\varphi_m^*|_k\gamma(\tau).
\end{equation*}
This series (if absolutely convergent), defines a function on the upper half plane, which transforms like a modular form of weight $k$ on $\Gamma_0(N)$. 

We single out two special cases ($m\in \N$), the first one the classical Poincar\'e series
\begin{equation*}
P(m,k,N;\tau):=\P(m,k,N,\exp(-2\pi my);\tau),
\end{equation*}
the second one the Maass-Poincar\'e series
\begin{equation*}
Q(-m,k,N;\tau):=\P(-m,2-k,N,\Phi_{-m};\tau),
\end{equation*}
where $\Phi_{-m}(y):=\calM_{1-\frac k2}(4\pi my)$ and $\calM_s(y)$ is a modified version of the Whittaker function $M_{\nu,\mu}$ (see e.g. \cite{Ono08}, p. 39).

We next recall the definition of Kloosterman sums
\begin{equation*}
K(m,n,c):=
\sum_{d \pmod{c}^*} e\left(\frac{m\overline
d+nd}{c}\right),
\end{equation*}
where $e(\alpha):=e^{2\pi i \alpha}$. The
 sum over $d$ runs through the primitive residue classes
modulo $c$, and $\overline d$ denotes the multiplicative inverse
of $d$ modulo $c$.

The Fourier expansions of the classical Poincar\'e series (see e.g. \cite{Iwaniecbook}),  as well as those of the Maass-Poincar\'e series treated in e.g. \cite{BringOPNAS, Bruinier, Fay,
Hejhal, Niebur1}, involve infinite sums of Kloosterman sums weighted by Bessel functions.
The classical case is treated in the following lemma.

\begin{lemma}\label{Hfourier} If $k\geq 2$ is even and $m, N\in \N$, then
the following are true for the Poincar\'e series
\begin{equation*}
P(m,k,N;\tau)=q^{m}+\sum_{n=1}^{\infty} a_m (n)q^n.
\end{equation*}

\smallskip
\noindent
(1) We have that
$P(m,k,N;\tau)\in S_k(\Gamma_0(N))$
and the coefficients are given by
\begin{displaymath}
a_m(n)=2\pi
(-1)^{\frac{k}{2}}\left(\frac{n}{m}\right)^{\frac{k-1}{2}} 
\sum_{\substack{c>0\\c\equiv 0\pmod{N}}} \frac{K(m,n,c)}{c}\cdot
J_{k-1} \left(\frac{4\pi \sqrt{mn}}{c}\right),
\end{displaymath}
where $J$ denotes the usual $J$-Bessel function.\\
\smallskip
\noindent
(2) We have that
$P(-m,k,N;\tau)\in M_k^{!}(\Gamma_0(N))$ and for
positive integers $n$ we have
\begin{displaymath}
a_{-m}(n)=2\pi(-1)^{\frac{k}{2}}\left(\frac{n}{m}\right)^{\frac{k-1}{2}}
 \sum_{\substack{c>0\\ c\equiv
0\pmod{N}}}\frac{K(-m,n,c)}{c}\cdot I_{k-1}\left(\frac{4\pi
\sqrt{|mn|}}{c}\right),
\end{displaymath}
where $I$ denotes the usual $I$-Bessel function.
\end{lemma}

\noindent For the case of Maass-Poincar\'e series we have the following result.

\begin{lemma}\label{Ffourier} If $k\geq 2$ is even and $m, N\geq 1$, then
$Q(-m,k,N;\tau)$ is in $H_{2-k}(\Gamma_0(N))$. Moreover, we have a Fourier expansion of the shape
$$
Q(-m,k,N;\tau) = (1-k) \left( \Gamma(k-1;4\pi my) - \Gamma(k-1) \right)
\, q^{-m} + \sum_{n\in\Z} c_m(n,y) \, q^n.
$$

\noindent (1) If $n<0$, then
\begin{displaymath}
\begin{split}
c_m(n,y)
  =2 \pi i^{k}  (1-k) \, &\Gamma(k-1;4 \pi |n| y)
  \left|\frac{n}{m}\right|^{\frac{1-k}{2}}\\
 &\ \ \ \ \times  \sum_{\substack{c>0\\c\equiv 0\pmod{N}}}
   \frac{K(-m,n,c)}{c}\cdot
 J_{k-1}\!\left(\frac{4\pi\sqrt{|mn|}}{c}\right).
\end{split}
\end{displaymath}

\noindent
(2) If $n>0$, then
$$
c_m(n,y)= - 2 \pi  i^k \Gamma(k)   \left( \frac{n}{m}
\right)^{\frac{1-k}{2}}
  \sum_{\substack{c>0\\c\equiv 0\pmod{N}}}
   \frac{K(-m,n,c)}{c}\cdot
 I_{k-1}\!\left(\frac{4\pi\sqrt{|mn|}}{c}\right).
$$

\noindent (3) If $n=0$, then
$$
c_m(0,y)=-(2\pi i)^{k} m^{k-1} \sum_{\substack{c>0\\c\equiv
0\pmod{N}}}
 \frac{K(-m,0,c)}{c^k}.
$$
\end{lemma}

These two kinds of Poincar\'e series are intimately related via the $\xi$-operator from the theory of harmonic Maass forms. The proof of this fact follows
easily from the formulas above.

\begin{lemma}\label{poincarerelationships}
If $k\geq 2$ is even and $m, N\geq 1$, then
\begin{displaymath}
\xi_{2-k}(Q(-m,k,N;\tau))= ( 4\pi)^{k-1} m^{k-1}(k-1)
P(m,k,N;\tau)\in S_k(\Gamma_0(N)).
\end{displaymath}
\end{lemma}

\subsection{Kloosterman sums}
Apart from their appearance in Fourier expansions of Poincar\'e series, Kloosterman sums are an important object of study for their own sake. We recall the following well-known multiplicativity property, which follows from the Chinese Remainder Theorem (see, e.g. equation (1.59) in \cite{IK}).
\begin{lemma}\label{mult}
For coprime positive integers $c$ and $d$ we have
\[K(m,n,c_1c_2)=K\left(m\overline{c_2},n\overline{c_2},c_1\right)K\left(m\overline{c_1},n\overline{c_1},c_2\right),\]
where $\overline{c_2}$ is the multiplicative inverse of $c_2$ modulo $c_1$ and vice versa.
\end{lemma}
This implies the following elementary, but important observation.
\begin{lemma}\label{Kloost}
Let $p$ be a prime and $m,n,c\in\N$ such that $p\nmid m$. Then we have
\[K\left(m,np,p^2c\right)=0.\]
\end{lemma}
\begin{proof}
By Lemma \ref{mult} it suffices to show the assertion for $c$ a power of $p$, say $cp^2=p^r$ with $r\geq 2$. We write $d\in(\Z/p^r\Z)^*$ as $d=s+\ell p^{r-1}$ with $s\in(\Z/p^{r-1}\Z)^*$ and $\ell\in\Z/p\Z$. It is easily checked that then $\overline d=\overline s-\overline{s}^2\ell p^{r-1}$, where $\overline{s}$ is the multiplicative inverse of $s$ modulo $p^{r-1}$. Thus we can write
 
\begin{align*}
K(m,pn,p^r)&=\sum\limits_{s \pmod{p^{r-1}}^*}\ \sum\limits_{\ell \pmod{p}} e\left(\frac{m\left(\overline{s}-\overline{s}^2\ell p^{r-1}\right)+np\left(s+\ell p^{r-1}\right)}{p^r}\right)\\
           &=\sum\limits_{s \pmod{p^{r-1}}^*} e\left(\frac{m\overline{s}+pns}{p^r}\right)\sum\limits_{\ell\pmod{s}}e\left(-\frac{\ell\overline{s}}{p}\right).
\end{align*}
The inner sum runs over all $p$th roots of unity and is therefore zero,
which in turn implies the claim.
\end{proof}

\subsection{The special case of newforms in Theorem~\ref{main}}

Here we apply some of the previous results to the harmonic Maass forms and weakly holomorphic modular forms which are pertinent to Theorem~\ref{main}.

\begin{theorem}\label{theonewform}
Let $f$ be as in Theorem \ref{main}. Then the following are all true.
\begin{enumerate}
\item We have that $f$ may be expressed as a finite linear combination of the form
$$f (\tau) =\sum\limits_{p\nmid m}\alpha_m P(m,k,N; \tau),
$$
with $\alpha_m\in\C$.
\item In terms of the linear combination in (1), if
$Q(\tau):=\sum\limits_{p\nmid m}\frac{\alpha_m}{m^{k-1}}Q(-m,k,N; \tau)$, then
$$\xi_{2-k}(Q)=( 4\pi)^{k-1}(k-1)f.
$$
\item If $Q^+(\tau)=\sum_n a_Q^+(n)q^n$, then $a_Q^+(pn)=0$ for all $n\in\N$.
\item We have that $D^{k-1}\left(Q^+\right)\in M_k^!(\Gamma_0(N))$, where $D:=\tfrac{1}{2\pi i}\tfrac{\partial}{\partial\tau}$ denotes the renormalized holomorphic derivative.
\end{enumerate}
\end{theorem}
\begin{proof}
(1) 
Since $f(\tau)=\sum\limits_{n=1}^\infty a_f(n)q^n$ is a newform on $\Gamma_0(N)$ with $p^2\mid N$, we know from Theorem~4.6.17. of \cite{Miyake} that the coefficient $a_f(p)$ vanishes, thus, by multiplicativity, we have $a_f(pn)=0$ for all $n\in\N$. By the Petersson coefficient formula
\[\langle f,P(m,k,N;\bullet)\rangle=\frac{(k-2)!}{(4\pi m)^{k-1}}a_f(m)\]
we therefore see that every Poincar\'e series $P(pm,k,N;\bullet)$ lies in the orthogonal complement of $f$, which implies our assertion.

Claim $(2)$ follows immediately from Lemma \ref{poincarerelationships}, claim $(3)$ is clear from Lemmas \ref{Ffourier} and \ref{Kloost}, and $(4)$ is an immediate consequence of Bol's identity
\[D^{k-1}(f)=(-4\pi)^{1-k}R_{2-k}^{k-1}(f),\]
where $R_k=2i\tfrac{\partial}{\partial \tau}+\tfrac{k}{y}$ denotes the classical Maass raising operator and $$R_k^n:=R_{k+2(n-1)}\circ ...\circ R_{k+2}\circ R_k,$$ whose extension to harmonic Maass forms is Theorem 1.1 of \cite{BOR}.
\end{proof}

\section{The generating function of shifted convolution $L$-values}

In this section we recall the main ideas and results from \cite{MO14}. For this, let $f \in S_{k}(\Gamma_0(N))$ be a cusp form with even weight
$k\geq 2$. Recall the definitions of the shifted convolution Dirichlet series $D(f,f,h;s)$ from \eqref{shiftedseries} and of the symmetrized shifted convolution Dirichlet series $\Dhat(f,f,h;s)$ from \eqref{symmetrizedseries}. The following theorem is a special case of a theorem proved by the second two authors.
\begin{theorem}[\cite{MO14}, Theorem 1.1]\label{Ltheorem}
The generating function $\L(f,f;\tau)$ from \eqref{LDefinition} is the sum of a weight $2$ mixed mock modular form  and 
a weight 2 weakly holomorphic quasimodular form on $\Gamma_0(N)$. More precisely, if $M_{f}$ denotes a harmonic Maass form whose shadow is $f$, then there exists a weakly holomorphic quasimodular form $F\in\widetilde{M}_2^{!}(\Gamma_0(N))$ such that
\[\L(f,f;\tau)=-\frac{1}{(k-2)!}M_{f}^+(\tau)\cdot f(\tau)+F(\tau).\]
\end{theorem}

\medskip
\noindent {{\it Two remarks.}}

\noindent (1) In \cite{MO14} the second two authors considered a more general definition of the symmetrized Dirichlet series for pairs of cusp forms $f_1$ and $f_2$ 
with weights $k_1\geq k_2$. 
For every non-negative integer $\nu$ with $\nu\leq \tfrac{k_1-k_2}{2}$
they investigated analogous generating functions $\L^{(\nu)}(f_1,f_2;\tau)$ which involve higher weight modular forms arising from $f_1$ and $f_2$.
For these generating functions they
proved a similar result where one replaces the mixed mock modular forms by the $\nu$th order Rankin-Cohen bracket $[M_{f_1}^+,f_2]_\nu$. 

\smallskip
\noindent (2)  If we can choose the function $M_{f}$ so that $M_{f}^+\cdot f$ is holomorphic on the complex upper half-plane and bounded near representatives of all cusps of $\Gamma_0(N)$, 
then we can find the function $F$ in Theorem \ref{Ltheorem} in the finite-dimensional space of quasimodular forms $\widetilde{M}_2(\Gamma_0(N))=M_2(\Gamma_0(N))\oplus\C E_2$. We note that such an $M_f$ usually does not exist.
\medskip

For concrete examples it is often handy to work with the following immediate corollary of Theorem \ref{Ltheorem} and Lemma \ref{poincarerelationships}.

\begin{corollary}\label{cuspidalgoodness} Suppose that $k\geq 2$ is even and $m$ is a positive integer.
If
$P(\tau):=P(m,k,N;\tau)\in S_{k}(\Gamma_0(N))$ and $Q(\tau):=
Q(-m,k,N;\tau)\in H_{2-k}(\Gamma_0(N))$, then
$$\L(P,P;\tau)= \frac{1}{m^{k-1}\cdot (k-1)!}\cdot Q^{+}(\tau) P(\tau) +F(\tau),
$$
where $F\in \widetilde{M}^{!}_2(\Gamma_0(N))$. Moreover, if $m=1$, then
$F\in \widetilde{M}_2(\Gamma_0(N))$.
\end{corollary}

\section{Proof of Theorem \ref{main}}

Let $f\in S_k(\Gamma_0(N))$ be a normalized newform of non-squarefree level $N$ whose coefficients lie in a number field $K$ and let $p$ be a prime with $p^2\mid N$. Further let $\calM_f(\tau)$ be a harmonic Maass form with holomorphic part 
\[\calM_f^+(\tau)=\sum_{n\gg -\infty} c^+(n)q^n\]
which is {\it good} for $f$ in the sense of \cite{BOR, GKO}. This means that $\calM_f$ has the following properties:
\begin{enumerate}
\item The principal part of $\calM_f^+$ at the cusp $\infty$ is in $K[q^{-1}]$.
\item The principal parts of $\calM_f^+$ at all other cusps is constant.
\item We have that $\xi_{2-k}(\calM_f)=\tfrac{f}{\Vert f\Vert^2}$, where $\Vert f\Vert$ denotes the Petersson norm of $f$.
\end{enumerate}
By Proposition 5.1 in \cite{BOR} we know that such an $\calM_f$ always exists. 

Now we know from Theorem \ref{Ltheorem} that
\begin{equation}\label{LcalM}
\L(f,f;\tau)=\frac{(4\pi)^{k-1}\Vert f\Vert^2}{(k-2)!}\calM_f^+(\tau)f(\tau)+\calQ_f(\tau),
\end{equation}
for a suitable weakly holomorphic quasimodular form $\calQ_f\in \widetilde{M}_2^!(\Gamma_0(N))$. Theorem 1.1 in \cite{GKO} yields that for every $\alpha\in\C$ with $\alpha-c^+(1)\in K$ the coefficients of the \emph{normalized mock modular form}
\[\calF_\alpha(\tau):=\calM_f^+(\tau)-\alpha\calE_f(\tau) \]
lie in the number field $K$ as well. Of course we can always choose $\alpha=c^+(1)$. In the case where $f$ has CM by a field of discriminant $d$, Theorem 1.3 in \cite{BOR} tells us that we can even choose $\alpha=0$ if we replace $K$ by $K(\zeta_{Nd})$, the $Nd$th cyclotomic field over $K$.

Since by Theorem 1.1 of \cite{BOR} $D^{k-1}\left(\calF_\alpha\right)\in M_k^!(\Gamma_0(N))$, we can write $\calF_\alpha$ as an Eichler integral
$\calF_\alpha=\calE_{g_\alpha}$
for some 
\[g_\alpha(\tau)=\sum\limits_{n\gg -\infty} a^{(\alpha)}(n)q^n\in M_k^!(\Gamma_0(N)).\]
Since $M_k^!(\Gamma_0(N))$ has a basis consisting of forms with integral Fourier coefficients, we know that the coefficients of $g_\alpha$ must have bounded denominators, in particular the $p$-adic valuation of $g_\alpha$, $v_p(g_\alpha):=\inf v_p\left(a^{(\alpha)}(n)\right)$, is bounded from below. By Theorem \ref{theonewform}, we know that $a^{(\alpha)}(n)=0$ whenever $p\mid n$ so that there can't be arbitrarily high negative powers of $p$ dividing the Fourier coefficients of $\calF_\alpha$ either, so that we can assume without loss of generality that the coefficients are $p$-integral, i.e, $v_p(\calF_\alpha)\geq 0$. 

For $n$ coprime to $p$ and $t\in\N$, we have the congruence 
\[n^{1-k}\equiv n^{(p-1)p^{t-1}+1-k}\pmod{p^t},\]
which implies that for every $r,t\in\N$ with $r\varphi(p^t)\geq k-1$, $\varphi$ denoting Euler's $\varphi$-function, we have that
\[\calF_\alpha\equiv D^{r(p-1)p^{t-1}-k+1}(g_\alpha)\pmod{p^t}.\]
Now $g_\alpha$ is a weakly holomorphic $p$-adic modular form of weight $k$ in the sense of Serre \cite{Serre} and by Th\'eor\`eme 5 (p. 211) of \cite{Serre}, we know that the operator $D$ maps $p$-adic modular forms of weight $k$ to $p$-adic modular forms of weight $k+2$. This follows immediately from the fact that the \emph{Serre derivative} $\theta(g):=D(g)-\tfrac{k}{12}E_2g$, maps modular forms of weight $k$ to modular forms of weight $k+2$ and that $E_2$ is a $p$-adic modular form of weight $2$. But this means that we have found a $p$-adically convergent sequence of $p$-adic modular forms which converges to $\calF_\alpha$, so $\calF_\alpha$ is a $p$-adic modular form as well. Since the sequence of weights $2-k+r(p-1)p^{t-1}$ $p$-adically converges to $2-k$, the weight of $\calF_\alpha$ as a $p$-adic modular form is indeed $2-k$, so that
\[\calF_\alpha f=\calM_f^+f-\alpha\calE_ff\]
becomes a $p$-adic modular form of weight $2$. If we now apply this in \eqref{LcalM}, our theorem follows.
\begin{remark}
Note that with the definitions from the above proof we can be more precise
 and make the constants explicit. For the non-CM case we can
 choose $\calL(f)=\calF_{c^+(1)}$ and therefore $\delta_1=\frac{(4\pi)^{k-1}\Vert f\Vert^2}{(k-2)!}\cdot p^\ell$ for a suitable exponent $\ell\in\N$ and $\delta_2=\delta_1\cdot c^+(1)$.
\end{remark}

\section{Example}\label{example}
 Consider the newform $f(\tau):=\eta(3\tau)^8\in S_4(\Gamma_0(9))$, where
 $\eta(\tau)$ is Dedekind's eta-function.
This form has complex multiplication by $\Q(\sqrt{-3})$, and is a multiple of the Poincar\'e
 series $P(1,4,9;\tau)$ since this space of cusp forms is one-dimensional. Theorem~\ref{main}
 applies for the prime $p=3$.
 Using a computer, one obtains the following
numerical approximations for the first few shifted convolution values.

\begin{displaymath}
\begin{array}{|c||c|c|c|c|c|}
\hline h & 3 & 6 & 9 & 12 & 15 \\
\hline \ & \ &   &   &   &   \\
\Dhat(f,f,h;3) & -10.7466\dots & 12.7931\dots & 6.4671\dots & -79.2777\dots & 64.2494\dots\\
\hline
\end{array}
\end{displaymath}

\medskip
\noindent
We note that $\Dhat(f,f,h;3)=0$ whenever $n$ is not a multiple of $3$.
We define real numbers $\beta, \gamma,$ and $\delta$ which are approximately
$$\beta:=\frac{(4\pi)^3}{2}\cdot \Vert P(1,4,9)\Vert^2=1.0468\dots,\ \ \  \gamma=-0.0796\dots,\ \ \ \delta=-0.8756\dots.
$$
As explained earlier, these real numbers arise naturally from the theory of Petersson inner products.
By Theorem 1.1 of \cite{MO14}, we have that
$$
\L(f,f;\tau)=\frac{f(\tau)Q^{+}(-1,4,9;\tau)}{\beta}+\gamma\left(1-24\sum_{n=1}^{\infty}
\sigma_1(3n)q^{3n}\right) + \delta\left(1+12\sum_{n=1}^{\infty}\sum_{\substack{d\mid 3n\\ 3\nmid d}}dq^{3n}\right).
$$
Since $f$ has complex multiplication, we have chosen $\delta_2=0$ in Theorem~\ref{main}. In particular, we have that $\delta_1:=1/\beta$,
$$
\calL_f(\tau):=Q^{+}(-1,4,9;\tau)=q^{-1}-\frac{1}{4}q^2+\frac{49}{125}q^5-\frac{3}{32}q^8-\dots,
$$
and 
$$\calQ_f(\tau)=\sum_{n=1}^{\infty} b_f(n)q^n:=\gamma\left(1-24\sum_{n=1}^{\infty}
\sigma_1(3n)q^{3n}\right) + \delta\left(1+12\sum_{n=1}^{\infty}\sum_{\substack{d\mid 3n\\ 3\nmid d}}dq^{3n}\right).
$$
We note that one can easily compute $\calL_f$ using the weight 4 weakly holomorphic modular form
$$
m(\tau):=\left(\frac{\eta(\tau)^3}{\eta(9\tau)^3}+3\right)^2\cdot \eta(3\tau)^8=q^{-1}+2q^2-49q^5+48q^8+\dots.
$$
It turns out that $\calL_f=-\calE_m$.

Theorem~\ref{main} proves that $\calL_f$ is a weight $-2$ weakly holomorphic $3$-adic modular form.
In fact, it turns out that $f\calL_f$ is a weight 2 cuspidal $3$-adic modular form. 
A straightforward argument proves that this product as a $q$-series is congruent to the constant 1 modulo 3.
This in turns implies that if
$h\neq 0$, then 
$$\widehat{D}(f,f,h;3)-b_f(h) \in \frac{3}{\beta}\cdot \Z_{(3)},
$$
where $\Z_{(3)}$ denotes the localization of $\Z$ at $3$, i.e., those rational numbers (in lowest terms) whose denominators are coprime to 3.
Moreover, we have the following higher congruences for every non-negative integer $n$
\begin{displaymath}
\begin{split}
 \widehat{D}(f,f,9n+6;3)-b_f(9n+6) &\in \frac{9}{\beta}\cdot \Z_{(3)},\\
 \widehat{D}(f,f,36n+30;3)-b_f(36n+30))&\in  \frac{27}{\beta}\cdot \Z_{(3)}.
  \end{split}
  \end{displaymath}
There are infinitely many such congruences modulo any power of 3. 

One striking consequence of the fact that $\calL_f$ is a cuspidal $3$-adic modular form is that
modulo any fixed power of 3, say $3^t$, ``almost all'' of the coefficients of this $p$-adic modular form are divisible by $3^t$.
Here we use ``almost all'' in the sense
of arithmetic density (see \cite{Serre74}). This is equivalent to the assertion that the rational numbers
$\beta(\widehat{D}(f,f,h;3)-b_f(h))$ are almost always multiples of any fixed power of 3.

To illustrate the phenomenon in general, we let
$$
\pi(3^t;X):=\frac{\# \{ 1 \leq h \leq X \ : \  \beta(\widehat{D}(f,f,h;3)-b_f(h))\equiv 0\pmod{3^t}\}}{X}.
$$
Here we illustrate these proportions for various small powers of 3. We stress that the convergence to 1 is extremely slow.

\medskip
\begin{displaymath}
\begin{array}{|c||c|c|c|c|c|}
\hline X & \pi(3;X)  & \pi(9;X) & \pi(27;X) & \pi(81;X) & \pi(243,;X) \\
\hline 3000 & 1&\ \ 0.912\dots & 0.784\dots & 0.705\dots & 0.676\dots \\
6000 & 1 &\ \ 0.917\dots & 0.792\dots& 0.711\dots & 0.679\dots\\
9000 & 1 &\ \ 0.920\dots & 0.798\dots& 0.716\dots & 0.680\dots\\
12000 & 1 &\ \ 0.922\dots & 0.800\dots& 0.718\dots & 0.681\dots\\
15000 & 1 &\ \ 0.923\dots & 0.803\dots& 0.720\dots & 0.683\dots\\
\vdots  & \vdots & \vdots & \vdots  & \vdots & \vdots \\
\infty   &    1       &  1&  1    &    1  &  1\\  
\hline
\end{array}
\end{displaymath}

\end{document}